\documentclass[11pt, twoside]{amsart}

\usepackage[utf8]{inputenc}
\usepackage[T1]{fontenc}
\usepackage{amssymb,amsmath,amstext}
\usepackage{hyperref, graphicx}

\newtheorem{theor}{Theorem}[section]
\newtheorem{lem}[theor]{Lemma}
\newtheorem{defin}[theor]{Definition}

\newtheorem{prop}[theor]{Proposition}

\newtheorem{exam}[theor]{Example}

\newtheorem{cor}[theor]{Corollary}
\newtheorem{rem}[theor]{Remark}

\newtheorem{fact}[theor]{Fact}

\newtheorem{assump}[theor]{Assumption}

\numberwithin{equation}{section}

\newcommand{\acl}{\mathrm{acl}}

\newcommand{\tp}{\mathrm{tp}}

\newcommand{\cl}{\mathrm{cl}}

\newcommand{\es}{\emptyset}

\newcommand{\rk}{\mathrm{rk}}

\newcommand{\nts}{\negthickspace}
\newcommand{\uhrc}{\nts \upharpoonright \nts}

\newcommand{\meq}{^{\mathrm{eq}}}

\newcommand{\mcA}{\mathcal{A}}
\newcommand{\mcB}{\mathcal{B}}

\newcommand{\mcG}{\mathcal{G}}

\newcommand{\mcM}{\mathcal{M}}
\newcommand{\mcN}{\mathcal{N}}

\newcommand{\mcP}{\mathcal{P}}

\newcommand{\mcS}{\mathcal{S}}

\newcommand{\mbC}{\mathbf{C}}

\newcommand{\mbbP}{\mathbb{P}}
\newcommand{\mbbN}{\mathbb{N}}

\newcommand{\indn}{\raisebox{-2pt}[5pt][0pt]{$\smile$} \hspace*{-6.8pt}\raisebox{3pt}[5pt][0pt]{$|^n$} }

\newcommand{\nindn}{\raisebox{-2pt}[5pt][0pt]{$\smile$} 
\hspace*{-6.8pt}\raisebox{3pt}[5pt][0pt]{$|^n$}\hspace*{-12pt}
\raisebox{3pt}[5pt][0pt]{$\diagup$} }

\newcommand{\indo}{\raisebox{-2pt}[5pt][0pt]{$\smile$} \hspace*{-6.8pt}\raisebox{3pt}[5pt][0pt]{$|^0$} }

\newcommand{\nindo}{\raisebox{-2pt}[5pt][0pt]{$\smile$} 
\hspace*{-6.8pt}\raisebox{3pt}[5pt][0pt]{$|^0$}\hspace*{-12pt}
\raisebox{3pt}[5pt][0pt]{$\diagup$} }

\newcommand{\rng}{\mathrm{rng}}

\title[A characterization of the finite model property]
{A characterization of the finite model property for $\omega$-categorical theories}

\author{Vera Koponen}

\address{Vera Koponen, Department of Mathematics, Uppsala University, Sweden.}

\email{vera.koponen@math.uu.se}

\date{12 October, 2026}

\begin{document}

\maketitle

\begin{abstract}
A first-order theory $T$ has the {\em finite model property} if  
every first-order sentence that is implied by $T$ has a finite model.
We give a characterization of when a complete $\omega$-categorical theory $T$ has the finite model property.
The characterization makes it possible to transfer stability theoretic results in \cite{KopInd}
to certain classes of finite structures.
This approach also gives information concerning properties that a sequence of finite ``base structures'' 
of ever larger size must have
in order for its random expansions to a larger vocabulary to have a chance to satisfy a convergence law
for first-order logic.
\end{abstract}

\section{Introduction}

\noindent
A first-order theory $T$ has the {\em finite model property (FMP)} if every first-order sentence 
$\varphi$ such that $T \models \varphi$ has a finite model. If in addition, for every such $\varphi$
there is a finite substructure $\mcA$
of some model of $T$ in which $\varphi$ is true, then we say that $T$ has the {\em finite submodel property (FSMP)}.
If $T$ has the F(S)MP then, in order to decide if $T \models \varphi$, it suffices to 
search among finite (sub)structures (of a model of $T$) which can, potentially, have consequences for the decidability
or computational complexity of the set of consequences of $T$.
Also, if $T$ has the FMP, then no first-order sentence that is a consequence of $T$ can 
express that ``all of my models are infinite''. So in particular, no consequence of a theory with the FMP
can express that some formula has the strict order property.

Results about the F(S)MP of a theory have typically been obtained via a 
rather detailed structural analysis of the models of the theory.
With this approach it has been proved that stable homogeneous structures have the FSMP 
(see e.g. Lachlan's survey \cite{Lach97}),
that $\omega$-categorical and $\omega$-stable theories (hence totally cateogorical theories)
have the FSMP (Cherlin, Harrington and Lachlan \cite{CHL}), and that smoothly approximable theories/structures have the FSMP (Kantor, Liebeck, and Macpherson \cite{KLM}, and Cherlin and Hrushovski \cite{CH}).
By probability theoretic arguments it follows that the random graph (also called the Rado graph), and more 
generally the random structure, has the FSMP (Glebskii, Kogan, Liogonkii and Talanov \cite{Gle} and
Fagin \cite{Fag}).
Later, the essential ideas of the probability theoretic arguments have been generalized to more general contexts in which
stability theoretic concepts have also been crucial to prove the FSMP (Djordjevi\'{c} \cite{Djo06a, Djo06b}, 
 Koponen \cite{Kop09}, 
 Ahlman \cite{Ahl}, 
and Donnay Hill \cite{Hill17}).
All above mentioned results concern $\omega$-categorical theories.
$\omega$-Categorical theories for which it is unknown if it has FMP include
the theory of the Fra\"{i}ss\'{e} limit of the class of all triangle-free graphs (see e.g. \cite{Che}),
and the theories of some structures obtained by Hrushovski style constructions (see e.g. \cite{Eva} and
Example~\ref{example 2 of first assumption} below).

In this study we will show that a complete $\omega$-categorical theory $T$ has the FMP if and only if 
$T$ consists of every sentence $\varphi$ that is true in all sufficiently large structures in a class $\mbC$
of finite structures with a nice behaviour with respect to a partial closure operator 
(Theorem~\ref{characterization of fmp}).
The characterization of the FMP of a complete $\omega$-categorical theory in terms of an associated class $\mbC$
of finite structures may be a useful guideline in the direction of understanding if concrete $\omega$-categorical 
theories have the FMP.
We also give some information about how the FSMP is related to the ``hereditary property for closed substructures''
(Theorem~\ref{HPC and fsmp}).
The original concern, eventually leading to this article, was to understand what kind of sequences of ever larger finite structures
would be suitable as ``base structures'' for random expansions to a richer vocabulary.
In this direction Corollary~\ref{0-1 law} 
is relevant for investigations of logical convergence laws of random expansions of finite
base structures, as explained in Remark~\ref{remark on expansions};
Example~\ref{expansions of the examples} also touches on convergence laws of expansions.

Now we turn to the second topic of this study, which as we will se, is related to the FMP.
Stability theory \cite{She} and its generalizations in many directions
(simple theories \cite{Kim-book}, rosy theories \cite{EO},
NIP theories \cite{Sim_book} etcetera) is, arguably, nowadays the
main theoretical framework with which first-order theories and their {\em infinite} models 
are analyzed.
Given this impact of stability theory on infinite structures it is of course natural to wonder if it is possible
to make sense of the ideas and methods of stability theory in the context of {\em finite} structures.
Since a complete first-order theory $T$ with a finite model has only one model up to isomorphism it is
natural to consider, not a complete first-order theory, but a class of finite structures $\mbC$ 
(which may possibly be the set of models of some incomplete first-order theory).
Work in this direction has been conducted by 
Baldwin \cite{Bal00} (a survey),
Baldwin and Benedikt \cite{BB},
Hyttinen \cite{Hyt00, Hyt02, Hyt15},
Djordjevi\'{c} \cite{Djo01},
Baldwin and Lessmann \cite{BL},
Koponen \cite{Kop09b},
and Donnay Hill \cite{Hill13, Hill14}.
Most of the mentioned studies have focused on amalgamation properties for the members of a class $\mbC$
of finite structures, often with respect to a fragment of first-order logic, like $L^k$
(first-order formulas with at most $k$ variables). 
Typically the assumed properties imply the existence of a nice limit structure or limit theory, with 
good stability theoretic properties provided that the members of $\mbC$ have good stability theoretic
properties relative to the fragment of first-order logic considered.

In this study we do not consider amalgamation properties, but an extension property relative to the 
partial closure operator. 
The assumed properties of the class $\mbC$ of finite structures will imply that for every first-order
sentence $\varphi$, either $\varphi$ or $\neg\varphi$ is true in all sufficiently large members of $\mbC$.
It also follows that the set $T_\mbC$ of all sentences that are true in all sufficiently large members of 
$\mbC$ is an $\omega$-categorical theory.
We will not assume anything of a ``stability theoretic nature'' about $\mbC$ and we do not get any stability theoretic
information about $T_\mbC$. But as shown in \cite{KopInd}, for all $n \in \mbbN$, a notion of `$n$-independence',
denoted $\indn$, is associated to every $\omega$-categorical theory, hence to $T_\mbC$.
$n$-Independence satisfies the usual properties of an independence relation except possibly for the symmetry
property. Since the relation `$\bar{x} \underset{\bar{z}}{\indn} \bar{y}$' is definable and
 the properties of $\indn$ (e.g. transitivity) for finite sets can be expressed by first-order formulas
it follows that all sufficiently large $\mcA \in \mbC$ will inherit a notion of $n$-independence
(on small subsets relative to the cardinality of $\mcA$) with all the
usual properties of an independence relation except possibly for the symmetry property.
But in \cite{KopInd} it is shown that the symmetry property of $\indn$ with respect to $\omega$-categorical
theories is equivalent to the ``exchange property for elements of $n$-rank 1''.
Again, using the definability of these properties and that $T_\mbC$ consists of all sentences that are true
in all sufficiently large structures in $\mbC$, it follows that we get a criterion for symmetry of $\indn$ in
structures in $\mbC$.
We explain the transfer of $\indn$ from $T_\mbC$ to $\mbC$ in
Section~\ref{independence in C}.

\subsection*{Basic notation and terminology}\label{preliminaries}

Structures are denoted by calligraphic letters, for example $\mcA, \mcB$, $\mcM$, and their universes by 
the corresponding noncalligraphich letters $A, B, M$. By $\acl_\mcM$ we dnote the algebraic closure operator in
the structure $\mcM$. If $A \subseteq M$ then $\mcM \uhrc A$ denotes the substructure of $\mcM$ that
is generated by $A$.
By a {\em closure operator} on a set $S$ we mean a function $\cl : \mcP(S) \to \mcP(S)$ such that
\begin{itemize}
\item $X \subseteq \cl(X)$ for all $X \in \mcP(S)$,
\item if $X \subseteq Y \in \mcP(S)$ then $\cl(X) \subseteq \cl(Y)$, and
\item $\cl(\cl(X)) = \cl(X)$ for all $X \in \mcP(S)$.
\end{itemize}
Suppose that, for some vocabulary $\tau$, $\mcM$ and $\mcN$ are $\tau$-structures and
$\cl_\mcM$ and $\cl_\mcN$ are closure operators on $M$ and $N$, respectively.
For $k \in \mbbN$, $\bar{a} \in M^k$ and $\bar{b} \in N^k$, we say that 
{\em $\bar{a}$ has the same closure type as $\bar{b}$}, denoted $(\mcM, \bar{a}) \equiv_{\cl} (\mcN, \bar{b})$,
if there is an isomorphism $f$ from $\mcM \uhrc \cl_\mcM(\bar{a})$ to $\mcN \uhrc \cl_\mcN(\bar{b})$
such that $f(\bar{a}) = \bar{b}$ and for all $C \subseteq \cl_\mcM(\bar{a})$ and $c \in \cl_\mcM(\bar{a})$,
$c \in \cl_\mcM(C)$ if and only if $f(c) \in \cl_\mcN(f(C))$.

$\mbbN$ denotes the set of natural numbers (including 0).
If $\bar{s}$ is a sequence $\rng(\bar{s})$ denotes the set of elements appearing as entries in $\bar{s}$,
and $|\bar{s}|$ is the length of $\bar{s}$.
For a set $X$ and $k \in \mbbN$, let $\mcP_k(X)$ denote the set of subsets of $X$ of cardinality at most $k$.

\section{Finite structures with a partial closure operator}
\label{finite structures with cl}

\noindent
In all of Section~\ref{finite structures with cl} 
we fix a signature $\tau$ and a class $\mbC$ of {\em finite} $\tau$-structures.
We will, step by step, add more assumptions and derive consequences from them.

\subsection{Finite structures with a partial closure operator}

{\bf \em For the rest of Section~\ref{finite structures with cl} we make the following assumption.}

\begin{assump}\label{assumptions on the closure}{\rm

\begin{enumerate}
\item There is no finite bound on the cardinalities of members of~$\mbC$.

\item For each $k < \omega$ there is $n_k < \omega$ such that if
 $\mcA \in \mbC$ and $|A| \geq n_k$, then there is $\cl_\mcA : \mcP_k(A) \to \mcP(A)$, 
 which we call a {\em partial closure operator (on A)},
such that 
\begin{enumerate}
\item $X \subseteq \cl_\mcA(X)$ for all $X \in \mcP_k(A)$,
\item if $X \subseteq Y \in \mcP_k(A)$ then $\cl_\mcA(X) \subseteq \cl_\mcA(Y)$, and
\item $\cl_\mcA(\cl_\mcA(X)) = \cl_\mcA(X)$ if $X \in \mcP_k(A)$ and $\cl_\mcA(X) \in \mcP_k(A)$.
\end{enumerate}

\item There is $F_\mbC : \mbbN \to \mbbN$ such that if $\mcA \in \mbC$, $|A| \geq n_k$ and $X \in \mcP_k(A)$,
then $|\cl_\mcA(X)| \leq F_\mbC(|X|)$.

\item For all $k < \omega$ there is a first-order $\tau$-formula $\xi_k(x_1, \ldots, x_{k+1})$ such that 
if $\mcA \in \mbC$, $|A| \geq n_k$ and $a_1, \ldots, a_{k+1} \in A$, then
$a_{k+1} \in \cl_\mcA(\{a_1, \ldots, a_k\})$ if and only if $\mcA \models \xi_k(a_1, \ldots, a_{k+1})$.

\item For all $\mcA \in \mbC$ and $X \subseteq A$ such that $\cl_\mcA(X)$ is defined,
$\mcA \uhrc \cl_\mcA(X)$ has universe $\cl_\mcA(X)$.

\item For all $k < \omega$ the class
\[
\big\{ \mcA \uhrc \cl_\mcA(X) : \mcA \in \mbC, X \in \mcP_k(A) \}
\]
has only finitely many isomorphism types.
\end{enumerate}
}\end{assump}

\noindent
If $\mcA \in \mbC$ and $a_1, \ldots, a_k \in A$ then `$\cl_\mcA(a_1, \ldots, a_k)$' means the same as
$\cl_\mcA(\{a_1, \ldots, a_k\})$.
Note that if $\tau$ has only relation symbols then condition~(5) above is automatically satisfied,
and if $\tau$ is finite then condition~(6) above is automatically satisfied.

Below a few examples of $\mbC$ that satisfy 
Assumption~\ref{assumptions on the closure}
are given.

\begin{exam}\label{example 1 of first assumption}  {\bf (DAGs)}
{\rm
By a {\em directed acyclic graph (DAG)} we mean a finite $\tau$-structure $\mcG$, 
where $\tau = \{R\}$ and $R$ is a binary relation symbol, such that there is no directed cycle in $\mcG$
and no loop (i.e. element $a$ with $R(a, a)$).
Choose any positive $s, t \in \mbbN$ and let $\mbC$ be the class of all finite DAGs $\mcG$ such that 
\begin{enumerate}
\item every vertex of $\mcG$ has at most $s$ parents, and
\item the height (the longest directed path) of $\mcG$ is at most $t$.
\end{enumerate}
For all $k \in \mbbN$ let $n_k = k$.
For all $\mcG \in \mbC$ and $X \subseteq A$ let $\cl_\mcG(X)$ be the union of $X$ and the
set of ancestors of elements of $X$.
Define $F_\mbC(x) = x \cdot s^t$. 
It is no straightforward to verify that $\mbC$ satisfies
Assumption~\ref{assumptions on the closure}.
}\end{exam}

\begin{exam}\label{example 2 of first assumption} {\bf (Structures with a predimension)} 
{\rm
Let $\tau = \{R\}$ where $R$ has arity $r \geq 2$.
Let $\alpha, \beta$ be positive integers.
For every finite $\tau$-structure $\mcA$ and $X \subseteq A$ define
\[
\delta_\mcA(X) = \alpha |X| - \beta|R^\mcA \cap X^r|.
\]
Let $f$ be a differentiable function from the nonnegative reals to the nonnegative reals such that
$\lim_{x\to\infty} f(x) = \infty$, $f'(x) \leq 1$,
$f'(x)$ is monotone decreasing, and $0 < f'(x) \leq 1/x$ for all nonnegative $x$.
Let $\mbC$ be the set of all finite $\tau$-structures $\mcA$ such that for all $X \subseteq A$,
$\delta(X) \geq f(|X|)$.
For every $\mcA \in \mbC$ define $\cl_\mcA : \mcP(A) \to \mcP(A)$ so that 
for all $X \subseteq A$, $X \subseteq \cl_\mcA(X)$ and for all $a \in A \setminus X$, 
$a \in \cl_\mcA(X)$ if and only if $\delta(\mcA \uhrc (X \cup \{a\})) < \delta(\mcA)$.
It follows from results in, for example, \cite{Eva} that conditions~(1)--~(3) of
Assumption~\ref{assumptions on the closure}
are satisfied.
Concerning condition~(4) of the same assumption:
As $\xi_k(x_1, \ldots, x_{k+1})$ we can take the disjunction of all the atomic diagrams 
$\theta_\mcA(x_1, \ldots, x_{k+1})$ of $\tau$-structures $\mcA \in \mbC$ with universe $A = \{a_1, \ldots, a_{k+1}\}$
such that $\delta_\mcA(A) < \delta_\mcA(\{a_1, \ldots, a_k\})$.
It follows that Assumption~\ref{assumptions on the closure}
is satisfied for this $\mbC$.
}\end{exam}

\begin{exam}\label{example of pregeometries 1} {\bf (Pregeometries)} {\rm
Let $\mbC$ be a class of finite pregeometries (matroids), so each member of $\mbC$ has the form
$\mcA = (A, \cl_\mcA)$ where $\cl_\mcA : \mcP(A) \to \mcP(A)$ is a closure operator,
and for all $X, Y \in \mcP(A)$,
if $a \in \cl_\mcA(X \cup \{b\}) \setminus \cl_\mcA(X)$ then $b \in \cl_\mcA(X \cup \{a\})$.
Also assume that there is no finite bound on the cardinalities of members of $\mbC$
and that for every $k \in \mbbN$ there is $m_k \in \mbbN$ such that $|\cl_\mcA(X)| \leq m_k$
whenever $\mcA \in \mbC$ and $X \in \mcP_k(A)$.
Let $\tau = \{P_0, P_1, P_2, \ldots, \}$, where $P_k$ is a relation symbol with arity $k+1$.
Then each $\mcA = (A, \cl_\mcA) \in \mbC$ can be viewed as a $\tau$-structure by
interpreting each $P_k$ so that 
$a_{k+1} \in \cl_\mcA(a_1, \ldots, a_k)$ if and only if $\mcA \models P_k(a_1, \ldots, a_{k+1})$.
Now Assumption~\ref{assumptions on the closure} is satisfied.
}\end{exam}

\begin{defin}\label{definition of closure type} 
{\rm
(i) A {\em closure diagram} in the variables $(x_1, \ldots, x_n)$ is a formula $\varphi(x_1, \ldots, x_n)$
such that
\begin{itemize}
\item $\varphi(x_1, \ldots, x_n)$ is a conjunction
of literals using only variables from $\{x_1, \ldots, x_n\}$ and of formulas
of the form $\xi_k(x_{i_1}, \ldots, x_{i_{k+1}})$ where $x_{i_1}, \ldots, x_{i_{k+1}} \in \{x_1, \ldots, x_n\}$
and $k < n$, or the negation of such a formula, 

\item for every atomic formula using only variables from $\{x_1, \ldots, x_n\}$ either itself or its negation
is a conjunct of $\varphi(x_1, \ldots, x_n)$, and

\item for every $k < n$ and all choices of $x_{i_1}, \ldots, x_{i_{k+1}} \in \{x_1, \ldots, x_n\}$,
either $\xi_k(x_{i_1}, \ldots, x_{i_{k+1}})$ or its negation is a conjunct of $\varphi(x_1, \ldots, x_n)$.
\end{itemize}

\noindent
(ii) A {\em closure type} in the variables $(x_1, \ldots, x_n)$
(with respect to $\mbC$)
is a formula of the following form that is satisfiable in some $\mcA \in \mbC$,
where $\pi(x_1, \ldots, x_n, y_1, \ldots, x_m)$ is a closure diagram:
\begin{align}\label{equation in def of cl tau-type}
\exists y_1, \ldots, y_m \Big( &\pi(x_1, \ldots, x_n, y_1, \ldots, y_m) \ \wedge \ 
\bigwedge_{i=1}^m \xi_n(x_1, \ldots, x_n, y_i)  \ \wedge \\
&\forall z \Big(\xi_n(x_1, \ldots, x_n, z) \rightarrow \Big[\bigvee_{i=1}^n z = x_i \ \vee \ 
\bigvee_{i=1}^m z = y_i \Big] \Big) \Big). \nonumber
\end{align}

\noindent
(iii) A formula of the form~(\ref{equation in def of cl tau-type})) allows the special case when the variables
`$y_1, \ldots, y_m$' and the parts `$\exists y_1, \ldots, y_m$' and `$\bigvee_{i=1}^m z = y_i$' are removed.
In this case we call the formula a {\em self-contained closure type}.
(Intuitively a self-contained closure type in the variables $x_1, \ldots, x_n$ expresses that 
the closure of $x_1, \ldots, x_n$ is itself.)
}\end{defin}

\noindent
Note that if the sequence of variables $x_1, \ldots, x_n$ in Definition~\ref{definition of closure type}
is {\em empty} then a closure type in this empty sequence of variables is a {\em sentence} $p$ which, if it
is satisfied in a sufficiently large $\mcA \in \mbC$, describes the 
isomorphism type and the ``closure structure'' of $\cl_\mcA(\es)$.

Note also that if~(\ref{equation in def of cl tau-type}) is a closure type (in the variables $(x_1, \ldots, x_n)$), 
then 
\[
\pi(x_1, \ldots, x_n, y_1, \ldots, y_m) \ \wedge \ 
\bigwedge_{i=1}^m \xi_n(x_1, \ldots, x_n, y_i)
\]
is a {\em self-contained} closure type in the variables $(x_1, \ldots, x_n, y_1, \ldots, y_m)$.
It follows that if $p(x_1, \ldots, x_n)$ is a closure type then there is a self-contained closure type
$p'(x_1, \ldots, x_n, y_1, \ldots, y_m)$ such that 
for every sufficiently large $\mcA \in \mbC$ the following hold:
\begin{enumerate}
\item[(a)] For all $a_1, \ldots, a_n \in A$, $\mcA \models p(a_1, \ldots, a_n)$ if and only if there are
$b_1, \ldots, b_m \in \cl_\mcA(a_1, \ldots, a_n)$ such that $\mcA \models p'(a_1, \ldots, a_n, b_1, \ldots, b_m)$.
\item[(b)] If $\mcA \models p'(a_1, \ldots, a_n, b_1, \ldots, b_m)$ then $b_1, \ldots, b_m \in \cl_\mcA(a_1, \ldots, a_n)$.
\end{enumerate}

\noindent
The following lemma is a straightforward consequence of the above definitions,
Assumption~\ref{assumptions on the closure}
and the definition (in Section~\ref{preliminaries}) of when two tuples have the same closure type.

\begin{lem}\label{tuples have the same closure type iff they satisfy the same formula}
For every $k \in \mbbN$, there are $s_k \in \mbbN$ and closure types
$\zeta_{k,i}(x_1, \ldots, x_k)$, for $i = 1, \ldots, s_k$, such that
\begin{enumerate}
\item for all sufficiently large $\mcA$ and all $a_1, \ldots, a_k \in A$, there is a unique $i$ such that
$\mcA \models \zeta_{k,i}(a_1, \ldots, a_k)$, and

\item for all sufficiently large $\mcA, \mcB \in \mbC$, all $a_1, \ldots, a_k \in A$, and all $b_1, \ldots, b_k \in B$,
$(\mcA, a_1, \ldots, a_k) \equiv_{\cl} (\mcB, b_1, \ldots, b_k)$ if and only if there is $i$ such that
$\mcA \models \zeta_{k,i}(a_1, \ldots, a_k)$ and $\mcB \models \zeta_{k,i}(b_1, \ldots, b_k)$.
\end{enumerate}
\end{lem}

\begin{defin}\label{definition of bounded closure type}{\rm
Let $p(\bar{x}, y)$ be a closure type where $\bar{x} = (x_1, \ldots, x_n)$.
Then $p(\bar{x}, y)$ is {\em $y$-bounded} if 
$p(\bar{x}, y) \models \xi_n(\bar{x}, y)$, or equivalently, if
for all sufficiently large $\mcA \in \mbC$, if $\mcA \models p(\bar{a}, b)$ then 
$b \in \cl_\mcA(\bar{a})$.
Otherwise we say that $p(\bar{x}, y)$ is {\em $y$-unbounded}.
}\end{defin}

\begin{defin}\label{definition of restriction of a closure type}{\rm
Let $p(\bar{x})$ be a closure type and let $\bar{y}$ be a subsequence of $\bar{x}$.
A closure type $q(\bar{y})$ will be called a {\em restriction of $p(\bar{x})$ to $\bar{y}$} if 
$p(\bar{x}) \models q(\bar{y})$.
It follows from the definition of closure type that if $q'(\bar{y})$ is another restriction of $p(\bar{x})$ to $\bar{y}$,
then $q(\bar{y})$ and $q'(\bar{y})$ are logically equivalent.
Therefore we will talk about {\em the} restriction of $p(\bar{x})$ to $\bar{y}$.
The restriction of $p(\bar{x})$ to $\bar{y}$ will be denoted by $p \uhrc \bar{y}$.
}\end{defin}

\noindent
{\bf \em For the rest of this section we make the following assumption:}

\begin{assump}\label{assumption that nonclosed types have realizations}
{\bf (Extension property)} {\rm
Suppose that $p(\bar{x}, y)$ is a closure type which is $y$-unbounded.
Let $q(\bar{x}) = p \uhrc \bar{x}$.
Then for all sufficiently large $\mcA \in \mbC$, 
$\mcA \models \forall \bar{x} \exists y \big(q(\bar{x}) \to p(\bar{x}, y)\big)$.
}\end{assump}

We consider some examples with respect to Assumption~\ref{assumption that nonclosed types have realizations}.

\begin{exam}\label{continuation of example 1} {\bf (DAGs)} {\rm
Let $\mbC$ and $\cl_\mcG$ for all $\mcG \in \mbC$ be defined as in
Example~\ref{example 1 of first assumption}, so 
Assumption~\ref{assumptions on the closure}
is satisfied.
In addition suppose that there is $\gamma \in \mbbN$ such that all $\mcG \in \mbC$ satisfy the following:
\begin{enumerate}
\item All leaves (vertices without a child) belong to layer $t$. 
(Vertices without parents, called roots, belong to layer 0. If a vertex $v$ 
has a parent on layer $k$ then $v$ belongs to layer $k+1$.)
\item Every vertex except for the roots has exactly $s$ parents all of which belong to the previous layer.
\item For every $k \in \mbbN$, if $|G|$ is sufficiently large 
then every choice of distinct $v_1, \ldots, v_s \in G$ from the same layer have at least $k$ common children.
\end{enumerate}
Then Assumption~\ref{assumption that nonclosed types have realizations}
is satisfied. 

A variation of this example is to replace~(3) with:
\begin{enumerate}
\item[(3')] For some fixed positive $r \in \mbbN$, 
every sequence of distinct $v_1, \ldots, v_s \in G$ from the same layer has exactly $r$ common children,
and for all $k \in \mbbN$, if $|G|$ is large enough then $\mcG$ has at least $k$ roots.
\end{enumerate}
Then we also modify $\cl_\mcG$ so that 
$\cl_\mcG(X) = X \cup A \cup B$ where $A$ is the set of ancestors of vertices in $X$, and $B$ contains all $v$
such all ancestors of $v$ belong to $A \cup X$.
Again one can verify that 
Assumption~\ref{assumption that nonclosed types have realizations}
is satisfied.
}\end{exam}

\begin{exam}\label{example of pregeometries 2} {\bf (Partially homogeneous pregeometries)} {\rm
Let $\mbC$ be as in Example~\ref{example of pregeometries 1},
so $\mbC$ is a class of finite pregeometries $\mcA = (A, \cl_\mcA)$ represented as $\tau$-structures where 
$\tau = \{P_0, P_1, P_2, \ldots, \}$ and $a_{k+1} \in \cl_\mcA(a_1, \ldots, a_k)$ if and only if
$\mcA \models P_k(a_1, \ldots, a_{k+1})$.
Suppose, in addition, that for all $k \in \mbbN$, if $\mcA \in \mbC$ is sufficiently large (relative to $k$),
$a_1, \ldots, a_k \in A$ is an independent sequence 
(i.e. $a_i \notin \cl_\mcA(\{a_1, \ldots, a_k\} \setminus \{a_i\})$ for all $i = 1, \ldots, k$) 
and $b_1, \ldots, b_k \in A$ is an
independent sequence, then $|\cl_\mcA(a_1, \ldots, a_k)| = |\cl_\mcA(b_1, \ldots, b_k)|$.
Since members of $\mbC$ can only express if an element is in the closure of a tuple of elements one can now
prove that 
Assumption~\ref{assumption that nonclosed types have realizations}
is satisfied.
}\end{exam}

\begin{exam}\label{expansions of the examples} {\bf (Expansions)} {\rm
Suppose that assumptions~\ref{assumptions on the closure}
and~\ref{assumption that nonclosed types have realizations}
are satisfied for $\mbC$.
Suppose that $\sigma \supset \tau$ and that $\mbC_\sigma$ is a 
class of $\sigma$-structures such that $\mbC = \{\mcA \uhrc \tau : \mcA \in \mbC_\sigma\}$. 
Then Assumption~\ref{assumptions on the closure} is satisfied for $\mbC_\sigma$.
But whether Assumption~\ref{assumption that nonclosed types have realizations}
is satisfied for $\mbC_\sigma$ depends on how $\mbC_\sigma$ is chosen.
One situation that can be of interest is when the universe of every 
$\mcA \in \mbC$ is $\{1, \ldots, n\}$ for some positive $n \in \mbbN$ and when a probability distribution $\mbbP_n$
is given on the set of $\mcB \in \mbC_\sigma$ with cardinality $n$.
In general we cannot expect that 
Assumption~\ref{assumption that nonclosed types have realizations}
is satisfied for $\mbC_\sigma$ in this situation.
But it may be the case that, for every $k \in \mbbN$, if $|\bar{x}| \leq k$ then, as $n \to \infty$,
the probability that a random $\mcB \in \mbC_\sigma$ with cardinality $n$ satisifies
$\forall \bar{x} \exists y \big(q(\bar{x}) \to p(\bar{x}, y)\big)$ tends to a positive number
(where $q(\bar{x})$ and $p(\bar{x}, y)$ are as in
Assumption~\ref{assumption that nonclosed types have realizations}).
In this case we can, for each $k \in \mbbN$ and for all large enough $n$, keep only those $\mcB \in \mbC_\sigma$ with cardinality $n$ that satisfy all 
$\forall \bar{x} \exists y \big(q(\bar{x}) \to p(\bar{x}, y)\big)$
with $|\bar{x}| \leq k$.
Then the resulting subclass of $\mbC_\sigma$ satisfies 
Assumption~\ref{assumption that nonclosed types have realizations}.
An example of the described situation appears in \cite{KT} where $\mbC$
is a class of trees of bounded height which is one possibility in Example~\ref{continuation of example 1}.
}\end{exam}

\begin{theor}\label{asymptotic quantifier elimination for tau-structures}
For every first-order formula $\varphi(\bar{x})$ that is satisfied in arbitrarily large members of $\mbC$ there 
are closure types $p_1(\bar{x}), \ldots, p_n(\bar{x})$
such that in all sufficiently large $\mcA \in \mbC$ we have
\[
\mcA \models \forall \bar{x} \bigg(\varphi(\bar{x}) \leftrightarrow \bigvee_{i=1}^n p_i(\bar{x})\bigg).
\]
\end{theor}

\noindent
{\bf Proof.}
The proof uses induction on the complexity of the first-order formula $\varphi(\bar{x})$.
Without loss of generality we can assume that every connective in $\varphi(\bar{x})$ (if any) is $\neg$ or $\vee$ and that 
every quantifier in $\varphi(\bar{x})$ (if any) is $\exists$.

First suppose that $\varphi(\bar{x})$ is an atomic formula.
Then let $p_1(\bar{x}), \ldots, p_k(\bar{x})$ be an enumeration of all (up to logical equivalence) 
closure types in the variables $\bar{x}$ (there are only finitely many) such that $p_i(\bar{x}) \models \varphi(\bar{x})$.
Then $\varphi(\bar{x})$ is equivalent to $\bigvee_{i = 1}^k p_i(\bar{x})$ in all sufficiently large $\mcA \in \mbC$.

Now suppose that $\varphi(\bar{x})$ has the form $\neg \psi(\bar{x})$.
By the induction hypothesis, $\psi(\bar{x})$ is equivalent, in all sufficiently large $\mcA \in \mbC$,
to $\bigvee_{i=1}^n p_i(\bar{x})$ where each $p_i(\bar{x})$ is a closure type.
Let $p'_1(\bar{x}), \ldots, p'_m(\bar{x})$ enumerate all (up to logical equivalence) closure types
that are not equivalent to any of $p_1(\bar{x}), \ldots, p_n(\bar{x})$.
Then $\varphi(\bar{x})$ is, in all sufficiently large $\mcA \in \mbC$, 
equivalent to $\neg \bigvee_{i=1}^n p_i(\bar{x})$ which in turn is (in all sufficiently large $\mcA \in \mbC$) 
equivalent to $\bigvee_{i=1}^m p'_i(\bar{x})$.

Next suppose that $\varphi(\bar{x})$ has the form $\varphi_1(\bar{x}) \vee \varphi_2(\bar{x})$.
By the induction hypothesis, there are closure types $p_1(\bar{x}), \ldots, p_n(\bar{x})$
and $p'_1(\bar{x}), \ldots, p'_m(\bar{x})$ such that $\varphi_1(\bar{x})$ is equivalent, in all sufficently large $\mcA \in \mbC$,
to $\bigvee_{i=1}^n p_i(\bar{x})$ and $\varphi_2(\bar{x})$ is equivalent, in all sufficiently large $\mcA \in mbC$,
to $\bigvee_{i=1}^m p'_i(\bar{x})$. 
Then $\varphi(\bar{x})$ is equivalent, in all sufficiently large $\mcA \in \mbC$, to
$\bigvee_{i=1}^n p_i(\bar{x}) \vee \bigvee_{i=1}^m p'_i(\bar{x})$ which is a disjunction of closure types.

Finally suppose that $\varphi(\bar{x})$ has the form $\exists y\psi(\bar{x}, y)$ where $\bar{x} = (x_1, \ldots, x_k)$.
By the induction hypothesis, $\psi(\bar{x}, y)$ is equivalent, in all sufficiently large $\mcA \in \mbC$,
to $\bigvee_{i = 1}^n p_i(\bar{x}, y)$ where all $p_i(\bar{x}, y)$ are closure types.
Then $\exists y \psi(\bar{x}, y)$ is equivalent, in all sufficiently large $\mcA \in \mbC$,
to $\bigvee_{i = 1}^n \exists y p_i(\bar{x}, y)$. 

Hence it suffices to show that for every closure type $p(\bar{x}, y)$,
$\exists y p(\bar{x}, y)$ is equivalent, in all sufficiently large $\mcA \in \mbC$, to $p \uhrc \bar{x}$.
Let $q(\bar{x}) = p \uhrc \bar{x}$ be the restriction of $p(\bar{x}, y)$ to $\bar{x}$, 
so $q(\bar{x})$ is a closure type.
We divide the argument into two cases. Either $p(\bar{x}, y)$ is $y$-bounded or it is $y$-unbounded.

First suppose that $p(\bar{x}, y)$ is $y$-bounded.
Since $p(\bar{x}, y) \models q(\bar{x})$ it follows that if $\mcA \not\models q(\bar{a})$ 
then $\mcA \not\models p(\bar{a}, b)$ for all $b \in A$ and hence $\mcA \not\models \exists y p(\bar{a}, y)$.
Now suppose that $\mcA \in \mbC$ and $\mcA \models q(\bar{a})$.
By the argument after Definition~\ref{definition of closure type} 
there is a self-contained closure type $p'(\bar{x}, y, \bar{z})$ such that 
$\forall \bar{x}, y \big(p(\bar{x}, y) \leftrightarrow \exists \bar{z} p'(\bar{x}, y, \bar{z})\big)$
holds in all sufficiently large $\mcA \in \mbC$.
Since $p(\bar{x}, y)$ is $y$-bounded we have $p(\bar{x}, y) \models \xi_k(\bar{x}, y)$ and it follows that
$\forall \bar{x} \big(q(\bar{x}) \leftrightarrow \exists y, \bar{z}  \ p'(\bar{x}, y, \bar{z})\big)$
holds in all sufficiently large $\mcA \in \mbC$.
So if $\mcA \in \mbC$ is sufficiently large and $\mcA \models q(\bar{a})$ then there are
$b, \bar{c}$ such that $\mcA \models p'(\bar{a}, b, \bar{c})$ and hence $\mcA \models p(\bar{a}, b)$.
As $q(\bar{x}) = p \uhrc \bar{x}$ it immeditately 
follows that if $\mcA \models p(\bar{a}, b)$ then $\mcA \models q(\bar{a})$.

Now suppose that $p(\bar{x}, y)$ is $y$-unbounded.
As $p(\bar{x}, y) \models q(\bar{x})$ it follows that, for all $\mcA \in \mbC$, if $\mcA \models p(\bar{a}, b)$
then $\mcA \models q(\bar{a})$.
If $\mcA \in \mbC$ and $\mcA \models q(\bar{a})$
it follows from Assumption~\ref{assumption that nonclosed types have realizations}
that if $\mcA$ is sufficiently large then $\mcA \models \exists y p(\bar{a}, y)$.
So in all sufficiently large $\mcA \in \mbC$, $\exists y p(\bar{x}, y)$ is equivalent to $q(\bar{x})$.
\hfill $\square$

\begin{cor}\label{0-1 law} {\bf (Zero-one law)} 
Let $()$ denote the empty sequence.
Suppose that for all sufficiently large $\mcA, \mcB \in \mbC$ we have $(\mcA, ()) \equiv_{\cl} (\mcB, ())$,
or equivalently, suppose that there is a closure type $p$ in the variables $()$, i.e. without
free variables, such that for all sufficiently large $\mcA \in \mbC$, $\mcA \models p$.
Then for every first-order sentence $\varphi$, either $\varphi$ is true in all sufficently large $\mcA \in \mbC$ or
$\varphi$ is false in all sufficiently large $\mcA \in \mbC$.
\end{cor}

\noindent
{\bf Proof.}
Recall that if a closure type in the variables $()$ is satisfied in sufficiently large 
$\mcA \in \mbC$ and $\mcB \in \mbC$, then there is an isomorphism from 
$\mcA \uhrc \cl_\mcA(\es)$ to $\mcB \uhrc \cl_\mcB(\es)$.
We adopt the assumptions of the corollary and let $\varphi$ be a first-order sentence.
By Theorem~\ref{asymptotic quantifier elimination for tau-structures}, $\varphi$ is equivalent,
in all sufficiently large $\mcA \in \mbC$, to $\bigvee_{i=1}^n p_i$ where all $p_i$ are closure types
in the variables $()$. 
From the assumption that for all sufficiently large $\mcA, \mcB \in \mbC$ we have $(\mcA, ()) \equiv_{\cl} (\mcB, ())$
it follows that there is a closure type $p$ (without free variables) such that
all sufficiently large $\mcA \in \mbC$ satisfy $p$.

If some $p_i$ is logically equivalent to $p$ then $\bigvee_{i=1}^n p_i$, and hence $\varphi$,
is true in all sufficiently large $\mcA \in \mbC$.
If no $p_i$ is logically equivalent to $p$ then $\bigvee_{i=1}^n p_i$, and hence $\varphi$,
is false in all sufficiently large $\mcA \in \mbC$.
\hfill $\square$

\begin{rem}\label{remark on expansions} {\bf (On random expansions)}
{\rm
Suppose that $\mbC = \bigcup_{n \in \mbbN} \mbC_n$ where $\mbC_n$ is a set of $\tau$-structures
with universe $\{1, \ldots, m_n\}$ for some $m_n$ and $\lim_{n\to\infty} m_n = \infty$.
Suppose that $\sigma \supset \tau$ is a vocabulary
and that $\mbC_n^\sigma$ is the set of all expansions to $\sigma$ of structures in $\mbC_n$.
Furthermore, let $\mbbP_n$ be a probability distribution on $\mbC_n^\sigma$.
If we aim at proving a convergence law for (say) all first-order sentences then, in particular,
$\lim_{n\to\infty} \mbbP_n(\varphi)$ must exist for every $\tau$-sentence $\varphi$.
Now suppose that each $\mbC_n$ is a singleton, containing a ``base structure'' which we randomly expand to $\sigma$
(and $\mbC_n^\sigma$ contains all such expansions).
If $\lim_{n\to\infty} \mbbP_n(\varphi)$ converges for a $\tau$-sentence $\varphi$, 
then the limit must be 0 or 1 (as each $\mbC_n$ is a singleton),
which means that $\varphi$ must either be true in all sufficiently large $\mcA \in \mbC$ or false in all
sufficiently large $\mcA \in \mbC$.
Thus assumptions \ref{assumptions on the closure}, \ref{assumption that nonclosed types have realizations},
and the additional assumption in Corollary~\ref{0-1 law}
give sufficient conditions for convergence to 0 or 1, with respect to $\mbbP_n$, for every $\tau$-sentence 
(provided that each $\mbC_n$ is a singleton).
For convergence laws for random expansions of various kinds of base structures see for example
\cite{ADGP, Bal03, DGH, KopDeg, KK, KT, Lyn}.
}\end{rem}

\subsection{The limit theory}

{\bf \em For the rest of the section let us add the following assumption.}

\begin{assump}\label{assumption about a unique 0-closure}{\rm
Suppose that for all sufficiently large $\mcA, \mcB \in \mbC$ we have $(\mcA, ()) \equiv_{\cl} (\mcB, ())$,
or equivalently, suppose that there is a closure type $p$ without
free variables such that all sufficiently large $\mcA \in \mbC$ satisfy $p$.
}\end{assump}

\noindent
By Corollary~\ref{0-1 law}, for every first-order $\tau$-sentence $\varphi$, either $\varphi$ is true in all 
sufficiently large $\mcA \in \mbC$ or $\varphi$ is false in all sufficiently large $\mcA \in \mbC$.
{\bf \em Let $T_\mbC$ be the set of all sentences $\varphi$ such that $\varphi$ is true in every sufficienlty large $\mcA \in \mbC$.}
It immediately follows that if $\varphi_1, \ldots, \varphi_n \in T_\mbC$, 
then $\varphi_1 \wedge \ldots \wedge \varphi_n \in T_\mbC$, so (by compactness) $T_\mbC$ is consistent.
From the choice of $T_\mbC$ it is evident that for every sentence $\varphi$,
either $\varphi$ or $\neg\varphi$ belongs to $T_\mbC$.
It is also clear that if $T_\mbC \models \varphi$ then $\varphi \in T_\mbC$.

\begin{theor}\label{T is omega-categorical}
$T_\mbC$ is $\omega$-categorical and has the FMP.
Moreover, for every formula $\varphi(\bar{x})$ there are closure types $p_1(\bar{x}), \ldots, p_m(\bar{x})$
such that $T_\mbC \models \forall \bar{x} \Big(\varphi(\bar{x}) \leftrightarrow \bigvee_{i=1}^m p_i(\bar{x})\Big)$.
\end{theor}

\noindent
{\bf Proof.}
It is immediate from the definition of $T_\mbC$ that it has the finite model property.
To prove that $T_\mbC$ is $\omega$-categorical it suffices to prove that, for every $n \in \mbbN^+$,
there are, up to equivalence in models of $T_\mbC$, only finitely many formulas in the free variables 
$\bar{x} = (x_1, \ldots, x_n)$.
Up to logical equivalence there are only finitely many closure types in the variables $\bar{x}$.
Hence it suffices to show that for every $\varphi(\bar{x})$ there are closure types $p_1(\bar{x}), \ldots, p_m(\bar{x})$
such that $T_\mbC \models \forall \bar{x} \Big(\varphi(\bar{x}) \leftrightarrow \bigvee_{i=1}^m p_i(\bar{x})\Big)$.
But this follows immediately from
Theorem~\ref{asymptotic quantifier elimination for tau-structures}, Corollary~\ref{0-1 law}
and the definition of $T_\mbC$.
\hfill $\square$

\begin{prop}\label{cl implies acl}
Let $\mcM \models T_\mbC$.
Then there is a closure operator $\cl_\mcM : \mcP(M) \to \mcP(M)$ such that, for all $k \in \mbbN$
and all $a_1, \ldots, a_{k+1} \in M$, $a_{k+1} \in \cl_\mcM(a_1, \ldots, a_k)$ if and only if
$\mcM \models \xi_k(a_1, \ldots, a_k, a_{k+1})$.
Moreover, for all $A \subseteq M$, $\cl_\mcM(A) \subseteq \acl_\mcM(A)$.
\end{prop}

\noindent
{\bf Proof.}
By the definition of $T_\mbC$, if $\varphi$ is true in all sufficiently large members of $\mbC$,
then $\varphi$ is true in $\mcM$. 
This together with Assumption~\ref{assumptions on the closure}
implies that there is $\cl'_\mcM : \mcP_{fin}(M) \to \mcP_{fin}(M)$ which is a closure operator
on {\em finite} subsets of $M$.

Define $\cl_\mcM : \mcP(M) \to \mcP(M)$ by 
\[
\cl_\mcM(A) = \bigcup_{\substack{B \subseteq A \\ B \text{ is finite}}} \cl'_\mcM(B)
\quad \text{ for all } A \in \mcP(M).
\]
Now it is straightforward to verify that $\cl_\mcM$ is a closure operator on $M$.

For the last claim, suppose that $A \subseteq M$, $b \in M$ and $b \in \cl_\mcM(A)$.
Then $b \in \cl_\mcM(a_1, \ldots, a_k)$ for some finite sequence $a_1, \ldots, a_k \in A$,
so $\mcM \models \xi_k(a_1, \ldots, a_k, b)$.
By  Assumption~\ref{assumptions on the closure} the first-order sentence 
which expresses that 
\[
\text{``for all $x_1, \ldots, x_k$, there are at most $F_\mbC(k)$ $y$ such that
$\xi_k(x_1, \ldots, x_k, y)$''}
\]
is true in all sufficiently large members of $\mbC$, and hence it is true in $\mcM$.
Hence, only finitely many members of $M$ satisfy $\xi_k(a_1, \ldots, a_k, y)$ so 
$b \in \acl_\mcM(a_1, \ldots, a_k) \subseteq \acl_\mcM(A)$.
\hfill $\square$

\medskip
Recall the definition (from Section~\ref{preliminaries})
of two tuples having the same closure type. In the next lemma we consider $\cl_\mcM$ from 
Proposition~\ref{cl implies acl}.

\begin{lem}\label{tuples have the same closure type iff they satisfy the same formula, in M}
Let $\mcM \models T_\mbC$. For every $k \in \mbbN$ let
$\zeta_{k,i}(x_1, \ldots, x_k)$, for $i = 1, \ldots, s_k$, be the closure types from
Lemma~\ref{tuples have the same closure type iff they satisfy the same formula}.
For all $k \in \mbbN$ the following hold:
\begin{enumerate}
\item For all $a_1, \ldots, a_k \in M$, there is a unique $i$ such that
$\mcM \models \zeta_{k,i}(a_1, \ldots, a_k)$.

\item For all $a_1, \ldots, a_k, b_1, \ldots, b_k \in M$,
$(a_1, \ldots, a_k)$ has the same closure type as $(b_1, \ldots, b_k)$ if and only if there is $i$ such that
$\mcM \models \zeta_{k,i}(a_1, \ldots, a_k)$ and $\mcM \models \zeta_{k,i}(b_1, \ldots, b_k)$.
\end{enumerate}
\end{lem}

\noindent
{\bf Proof.}
This follows from the definitions of closure types, tuples having the same closure type, and since
$a_{k+1} \in \cl_\mcM(a_1, \ldots, a_k)$ if and only if $\mcM \models \xi_k(a_1, \ldots, a_{k+1})$.
\hfill $\square$

\begin{defin}\label{definition of HPC}{\rm
Let us say that $\mbC$ has the {\em hereditary property for closed substructures (HPC)}
if whenever $\mcA \in \mbC$, $\mcB \subseteq \mcA$, $\cl_\mcA(B)$ is defined
(which it is if $|A|$ is sufficiently large compared to $|B|$), and $\cl_\mcA(B) = B$,
then there is $\mcB' \in \mbC$ which is isomorphic to $\mcB$.
}\end{defin}

\begin{lem}\label{technical lemma about HPC}
Suppose that $\mbC$ has the HPC and that $\mcM \models T_\mbC$.
If $\mcA$ is a finite substructure of $\mcM$ and $\cl_\mcM(A) = A$
then there is $\mcA' \in \mbC$ which is isomorphic to $\mcA$.
\end{lem}

\noindent
{\bf Proof.}
Suppose that $\mbC$ has the HPC, $\mcM \models T_\mbC$ and that $\mcA$ is a finite substructure
of $\mcM$ such that $\cl_\mcM(A) = A$.
Let $A = \{a_1, \ldots, a_k\}$.
By Lemma~\ref{tuples have the same closure type iff they satisfy the same formula, in M},
there is $i \in \{1, \ldots, s_k\}$ and a $\tau$-closure type $\zeta_{k, i}(x_1, \ldots, x_k)$
such that $\mcM \models \zeta_{k, i}(a_1, \ldots, a_k)$.
Let $\varphi_\mcA$ be the sentence
\[
\exists x_1, \ldots, x_k \zeta_{k, i}(x_1, \ldots, x_k).
\]
Since $\mcM \models \varphi_\mcA$ we have $\varphi_\mcA \in T_\mbC$ and therefore $\varphi_\mcA$
holds in all sufficiently large members of $\mbC$.
Let $\mcB \in \mbC$ be such that $\mcB \models \varphi_\mcA$.
By the form of $\varphi_\mcA$ it follows that there is $\mcA' \subseteq \mcB$
such that $\mcA' \cong \mcA$ and $\cl_\mcB(A') = A'$.
Since $\mbC$ has the HPC it follows that there is $\mcA'' \in \mbC$ such that $\mcA'' \cong \mcA'$
and hence $\mcA'' \cong \mcA$.
\hfill $\square$

\begin{theor}\label{HPC and fsmp}
If $\mbC$ has the HPC then $T_\mbC$ has the FSMP.
\end{theor}

\noindent
{\bf Proof.}
Suppose that $\mbC$ has the HPC and let $\varphi \in T_\mbC$.
Let $\mcM \models T_\mbC$. The model does not matter for the argument, but we can, for example, 
let $\mcM$ be the unique (up to isomorphism) countable model of $T_\mbC$.
We will show that $\varphi$ is true in a finite substructure of $\mcM$.

Since $\varphi \in T_\mbC$ it follows that there is $n_0 \in \mbbN$ such that $\mcA \models \varphi$
for all $\mcA \in \mbC$ such that $|A| \geq n_0$.
By Assumption~\ref{assumptions on the closure}
there is $n_1 \in \mbbN$ such that if $\mcA \in \mbC$ is sufficiently large relative to $n_0$, 
$B \subseteq A$ and $|B| \leq n_0$,
then $|\cl_\mcA(B)| \leq n_1$.
Let $\mcA_1, \ldots, \mcA_s$ enumerate all (up to isomorphism) $\tau$-structures $\mcA$ such that 
$n_0 \leq |A| \leq n_1$ and there is 
$\mcB \in \mbC$ such that $\mcA \subseteq \mcB$ and $\cl_\mcB(A) = A$.
By the choice of $n_0$ and HPC we have  $\mcA_i \models \varphi$ for all $i = 1, \ldots, s$.
For each $i = 1, \ldots, s$, let $\theta_i$ be a sentence which expresses that 
``there is a substructure which is ismorphic to $\mcA_i$'', so $\theta_i$ can be 
$\exists x_1, \ldots, x_m \psi(x_1, \ldots, x_m)$ where $\psi(x_1, \ldots, x_m)$ is a complete $\tau$-diagram
that is satisfied, in $\mcA_i$, by some enumeration of all elements of $A_i$.

From the choice of $\mcA_1, \ldots, \mcA_s$ it follows that for every $\mcB \in \mbC$ such that $|B| \geq n_0$
and every choice of distinct $b_1, \ldots, b_{n_0} \in B$,
$\mcB \uhrc \cl_\mcB(b_1, \ldots, b_{n_0})$ is isomorphic to $\mcA_i$ for some $i$
and hence $\mcB \models \bigvee_{i=1}^s \theta_i$.
Thus all sufficiently large members of $\mbC$ satisfy $\bigvee_{i=1}^s \theta_i$, so 
$\bigvee_{i=1}^s \theta_i \in T_\mbC$ and therefore $\mcM \models \bigvee_{i=1}^s \theta_i$.
Let $i$ be such that $\mcM \models \theta_i$.
Then $\mcM$ has a substructure which is isomorphic to $\mcA_i$ where we recall  that $\mcA_i \models \varphi$.
Thus $T_\mbC$ has the FSMP.
\hfill $\square$

\medskip

\noindent
We now consider the following stronger variant of 
Assumption~\ref{assumption that nonclosed types have realizations},
which will be relevant in the remaining sections:

\begin{assump}\label{stronger assumption that nonclosed types have realizations}
{\bf (Multiextension property)} {\rm
Suppose that $p(\bar{x}, y)$ is a closure type which is $y$-unbounded.
Let $q(\bar{x}) = p \uhrc \bar{x}$.
For every $m \in \mbbN$ there is $n_m$ such that if $\mcA \in \mbC$, $|A| \geq n_m$,
and $\mcA \models q(\bar{a})$, then there are distinct $b_1, \ldots, b_m \in A$ such
that $\mcA \models p(\bar{a}, b_i)$ for all $i = 1, \ldots, m$.
}\end{assump}

\begin{prop}\label{cl and acl coincide in limit structures}
Let $\mcM \models T_\mbC$ and let $\cl_\mcM : \mcP(M) \to \mcP(M)$ be the
closure operator from Lemma~\ref{cl implies acl}.
Assumption~\ref{stronger assumption that nonclosed types have realizations}
implies that for all $A \subseteq M$, $\cl_\mcM(A) = \acl_\mcM(A)$.
\end{prop}

\noindent
{\bf Proof.}
From 
Lemma~\ref{cl implies acl}
we have $\cl_\mcM(A) \subseteq \acl_\mcM(A)$, so only the converse inclusion needs to be proved.
We prove that if $b \notin \cl_\mcM(A)$ then $b \notin \acl_\mcM(A)$.
Since $b \in \cl_\mcM(A)$ implies that $b \in \cl_\mcM(A')$ for some finite $A' \subseteq A$,
and the same with $\acl_\mcM$ in place of $\cl_\mcM$, it follows that we only need to
show that for finite  $A \subseteq M$, if $b \notin \cl_\mcM(A)$ then $b \notin \acl_\mcM(A)$.

So let $\bar{a}$ be a finite sequence and suppose that $b \notin \cl_\mcM(\bar{a})$.
Let $p(\bar{x}, y)$ be the (unique up to equivalence in $T_\mbC$) closure type that is satisfied by 
$\bar{a}b$ in $\mcM$. 
Let $q(\bar{x}) = p \uhrc \bar{x}$.
As $b \notin \cl_\mcM(\bar{a})$ it follows that $p(\bar{x}, y)$ is $y$-unbounded.
It follows from Assumption~\ref{stronger assumption that nonclosed types have realizations}
that, for every $m \in \mbbN$, the sentence $\psi_m$ expressing
\[
\text{``for all $\bar{x}$, if $q(\bar{x})$ then there are distinct $y_1, \ldots, y_m$ such that $p(\bar{x}, y_i)$ for all $i$''}
\]
is true in all sufficiently large members of $\mbC$.
Hence $\mcM \models \psi_m$ for all $m \in \mbbN$ and therefore
there are infinitely many $b \in M$ such that $\mcM \models p(\bar{a}, b)$.

Let $\varphi(\bar{x}, y)$ be any formula with free variables $\bar{x}$ and $y$.
It follows from Theorem~\ref{T is omega-categorical}
that either 
\[
T_\mbC \models \forall \bar{x}, y \big(p(\bar{x}, y) \rightarrow \varphi(\bar{x}, y)\big) \text{ or }
T_\mbC \models \forall \bar{x}, y \big(p(\bar{x}, y) \rightarrow \neg\varphi(\bar{x}, y)\big).
\]
It follows that either $\varphi(\bar{a}, y)$ is not satisfied by any element of $\mcM$ or
$\varphi(\bar{a}, y)$ is satisfied by infinitely many elements of $\mcM$.
As $\varphi(\bar{x}, y)$ can be any formula with free variables $\bar{x}$ and $y$
it follows that $b \notin \acl_\mcM(\bar{a})$.
\hfill $\square$

\section{Characterization of FMP for $\omega$-categorical theories}\label{Characterization of FMP}

\noindent
For a complete $\omega$-categorical theory $T$ we define a {\em closure type (with respect to $T$)}
just as we did in 
Definition~\ref{definition of closure type}
but now we assume that, for all $k \in \mbbN$, $\xi_k(x_1, \ldots, x_{k+1})$ expresses (in all models of $T$)
that $x_{k+1}$ belongs to the algebraic closure of $\{x_1, \ldots, x_k\}$.

\begin{theor}\label{characterization of fmp}
Let $T$ be a complete $\omega$-categorical theory.\\
(i) $T$ has the FMP if and only if there is a class $\mbC$ of finite structures such
that 
\begin{enumerate}
\item $\mbC$ satisfies 
assumptions~\ref{assumptions on the closure},~\ref{assumption about a unique 0-closure},
and~\ref{stronger assumption that nonclosed types have realizations}, 
and 
\item $T$ is the set of all sentences $\varphi$ such that $\varphi$ is true in every sufficiently large member of $\mbC$.
\end{enumerate}
(ii) If $T$ has the FMP then, for every first-order formula $\varphi(\bar{x})$ that is consistent with $T$, there are 
closure types $p_1(\bar{x}), \ldots, p_s(\bar{x})$ such that $\varphi(\bar{x})$ is equivalent to 
$\bigvee_{i=1}^s p_i(\bar{x})$ in all models of $T$.
\end{theor}

\noindent
{\bf Proof.}
Let $T$ be a complete $\omega$-categorical theory.
We first prove (i).
Suppose that  $\mbC$ is a class of finite structures such that~(1) and~(2) hold.
From condition~(1) and Corollary~\ref{0-1 law} it follows that every first-order sentence is either true
in all sufficiently large members of $\mbC$ or false in all sufficiently large members of $\mbC$.
Now condition~(2) implies that $T$ has the FMP.

It remains to prove that if $T$ has the FMP then there is a class $\mbC$ of finite structures such that
conditions~(1) and~(2) hold.
Suppose that $T$ has the FMP.
Let $\mcM \models T$ be countable.
Note that since $T$ is $\omega$-categorical the following conditions hold:

\begin{enumerate}
\item[(I)]
\begin{enumerate}
\item $X \subseteq \acl_\mcM(X)$ for all $X \in \mcP(X)$,
\item if $Y \subseteq X \in \mcP(M)$ then $\acl_\mcM(Y) \subseteq \acl_\mcM(X)$, and
\item $\acl_\mcM(\acl_\mcM(X)) = \acl_\mcM(X)$ for all $X \in \mcP(M)$.
\end{enumerate}

\item[(II)] There is $F_T : \mbbN \to \mbbN$ such that if $X \in \mcP(M)$,
then $|\acl_\mcM(X)| \leq F_T(|X|)$.

\item[(III)] For all $k \in \mbbN$ there is a first-order $\tau$-formula $\xi_k(x_1, \ldots, x_{k+1})$ such that 
if $a_1, \ldots, a_{k+1} \in M$, then
\[
\text{$a_{k+1} \in \acl_\mcM(\{a_1, \ldots, a_k\})$ if and only if $\mcM \models \xi_k(a_1, \ldots, a_{k+1})$.}
\]
\end{enumerate}

Let $\cl_\mcM(A) = \acl_\mcM(A)$ for all $A \subseteq M$, so the above holds with $\cl_\mcM$ in place of $\acl_\mcM$.
From~(III) it follows that for each $m \in \mbbN$ there is a sentence $\varphi'_m$ which
expresses, in $\mcM$, that~(I) and~(II) hold if $|X| \leq m$.
We can define the notion of a closure type in, say, variables $\bar{x}$ just as we did in
Section~\ref{finite structures with cl} by using~(III). 
It follows that if a closure type $p(\bar{x})$ is satisfied by a finite sequence $\bar{a}$ of
elements from $M$, then $p(\bar{x})$ describes (a) the isomorphism type of $\mcM \uhrc \acl_\mcM(\bar{a})$
and (b) for all $B \subseteq \acl_\mcM(\bar{a})$ and $b \in \acl_\mcM(\bar{a}),$ if $b \in \acl_\mcM(B)$ or not.

Let us say that a closure type $p(x_1, \ldots, x_k, y)$ (with respect to $T$) is {\em $y$-bounded} if and only if
$p(x_1, \ldots, x_k, y) \models \xi_k(x_1, \ldots, x_k, y)$, which is equivalent to say that
if $\mcM \models p(a_1, \ldots, a_k, b)$ then $b \in \acl_\mcM(a_1, \ldots, a_k)$.
Otherwise it is $y$-unbounded.

It follows that there is a sentence $\varphi''_m$ expressing, in $\mcM$, that 
\begin{enumerate}
\item[] ``if $|\bar{x}| \leq m$, $p(\bar{x}, y)$ is a closure type which is $y$-unbounded, 
and $\bar{a}$ satisfies $q(\bar{x}) = p \uhrc \bar{x}$, then there are distinct $b_1, \ldots, b_m$
such that $p(\bar{a}, b_i)$ holds for all $i$''.
\end{enumerate}
Let $\varphi_m = \varphi'_m \wedge \varphi''_m$.
Note that $\varphi_{m+1} \models \varphi_m$ and $\mcM \models \varphi_m$
(and hence $T \models \varphi_m$)  for all $m \in \mbbN$.

Since $T$ has the FMP it follows that each $\varphi_m$ has a finite model.
Let 
\[
\mbC = \{\mcA : \text{ $\mcA$ is a finite model of $\varphi_m$ for some $m \in \mbbN$}\}.
\]
For every $k \in \mbbN$, $\mcA \in \mbC$, and $a_1, \ldots, a_{k+1}$, let
`$a_{k+1} \in \cl_\mcA(a_1, \ldots, a_k)$' mean that $\mcA \models \xi_k(a_1, \ldots, a_{k+1})$.
From the choice of $\varphi_m$ it follows that conditions~(1)--(4)
of Assumption~\ref{assumptions on the closure}
are satisfied by $\mbC$.
Since $T$ is $\omega$-categorical it follows that also conditions (5)--(6) of 
Assumption~\ref{assumptions on the closure}
are satisfied by $\mbC$.
Moreover, (as $\varphi_m \models \varphi''_m$) 
Assumption~\ref{stronger assumption that nonclosed types have realizations},
and hence Assumption~\ref{assumption that nonclosed types have realizations},
also holds for $\mbC$.

It follows from Theorem~\ref{asymptotic quantifier elimination for tau-structures}
that,
\begin{enumerate}
\item[] for every formula $\varphi(\bar{x})$ that is satisfied in arbitrarily large members of $\mbC$
there are closure types $p_1(\bar{x}), \ldots, p_s(\bar{x})$
such that for every sufficiently large $\mcA \in \mbC$,
$\mcA \models \forall \bar{x} \Big(\varphi(\bar{x}) \leftrightarrow \bigvee_{i=1}^s p_i(\bar{x})\Big)$.
\end{enumerate}

Since the isomorphism type of $\mcM \uhrc \acl_\mcM(\es)$ is determined by $T$ it follows that 
there is a unique (up to equivalence) closure type, say $p_\es$, without free variables  with respect to $T$ .
By the definition of $\mbC$, $p_\es$ will also be the unique closure type without free variables with respect to $\mbC$.
Hence Assumption~\ref{assumption about a unique 0-closure} is satisfied by $\mbC$.
It follows from Corollary~\ref{0-1 law}
that for every sentence $\varphi$, either it is true in all sufficiently large $\mcA \in \mbC$ or it is
false in all sufficiently large $\mcA \in \mbC$.
Let $T_\mbC$ be the set of all sentences $\varphi$ such that $\varphi$ holds in all sufficiently large $\mcA \in \mbC$.
As noted in Section~\ref{finite structures with cl} $T_\mbC$ is consistent, complete and $\omega$-categorical.

Let $\mcM_\mbC \models T_\mbC$ be countable and recall that $\mcM \models T$ is countable.
Since $\mbC$ satisfies 
Assumption~\ref{stronger assumption that nonclosed types have realizations}
it follows from Proposition~\ref{cl and acl coincide in limit structures}
that $\cl_{\mcM_\mbC}$ 
(defined by $a_{k+1} \in \cl_{\mcM_{\mbC}}(a_1, \ldots, a_k)$ if and only if
$\mcM_\mbC \models \xi_k(a_1, \ldots, a_{k+1})$)
is the same as $\acl_{\mcM_\mbC}$.
We will show that $\mcM \cong \mcM_\mbC$ from which it follows that $T = T_\mbC$.
To show that $\mcM \cong \mcM_\mbC$ we use a back-and-forth argument to construct isomorphisms
between (larger and larger) finite substructures of $\mcM$ and $\mcM_\mbC$, the union of which will
be the desired isomorphism from $\mcM$ to $\mcM_\mbC$.
Recall that if $(\mcM, \bar{a}) \equiv_{\cl} (\mcM_\mbC, \bar{b})$
then there is an isomorphism from $\mcM \uhrc \cl_\mcM(\bar{a})$
to $\mcM_\mbC \uhrc \cl_{\mcM_\mbC}(\bar{b})$ which maps $\bar{a}$ to $\bar{b}$.
Hence it suffices to show the following:

\medskip

\noindent
{\bf Claim.}
Suppose that $\bar{a} = (a_1, \ldots, a_k) \in M^k$,
$\bar{b} = (b_1, \ldots, b_k) \in (M_\mbC)^k$, and
$(\mcM, \bar{a}) \equiv_{\cl} (\mcM_\mbC, \bar{b})$.
If $a_{k+1} \in M$ ($b_{k+1} \in M_\mbC$) then there is $b_{k+1} \in M_\mbC$ ($a_{k+1} \in M$)
such that 
$(\mcM, \bar{a}a_{k+1}) \equiv_{\cl} (\mcM_\mbC, \bar{b}b_{k+1})$.

\medskip

\noindent
{\em Proof of the claim.}
Suppose that 
$\bar{a} = (a_1, \ldots, a_k) \in M^k$,
$\bar{b} = (b_1, \ldots, b_k) \in (M_\mbC)^k$, and
$(\mcM, \bar{a}) \equiv_{\cl} (\mcM_\mbC, \bar{b})$.
Then there is a closure type $q(\bar{x})$, where $\bar{x} = (x_1, \ldots, x_k)$ such that $\mcM \models q(\bar{a})$
and $\mcM_\mbC \models q(\bar{b})$.
Let $a_{k+1} \in M$. (The case when we pick $b_{k+1} \in M_\mbC$ is treated similarly.)

First suppose that $a_{k+1} \in \acl_\mcM(\bar{a})$.
From the definition of a closure type it follows that
there is a closure type $p(\bar{x}, x_{k+1})$ such that 
$p(\bar{x}, x_{k+1}) \models \xi_k(\bar{x}, x_{k+1})$,
$q(\bar{x}) \models \exists x_{k+1} p(\bar{x}, x_{k+1})$,
and $\mcM \models p(\bar{a}, a_{k+1})$.
As $\mcM_\mbC \models q(\bar{b})$ it follows that there is $b_{k+1} \in M_\mbC$ such that 
$\mcM_\mbC \models p(\bar{b}, b_{k+1})$.
Hence $(\mcM, \bar{a}a_{k+1}) \equiv_{\cl} (\mcM_\mbC, \bar{b}b_{k+1})$.

Now suppose that $a_{k+1} \notin \acl_\mcM(\bar{a})$.
Let $p(\bar{x}, x_{k+1})$ be the (unique up to equivalence) closure type that is satisfied by $\bar{a}a_{k+1}$.
Then $p(\bar{x}, x_{k+1}) \models \neg\xi_k(\bar{x}, x_{k+1})$
so $p(\bar{x}, x_{k+1})$ is $x_{k+1}$-unbounded.
Since $\mbC$ satisfies Assumption~\ref{stronger assumption that nonclosed types have realizations},
hence also Assumption~\ref{assumption that nonclosed types have realizations},
it follows that if $\mcA \in \mbC$ is large enough then 
$\mcA \models \forall \bar{x} \exists x_{k+1} \big(q(\bar{x}) \rightarrow p(\bar{x}, x_{k+1})\big)$.
Hence $T_\mbC \models \forall \bar{x} \exists x_{k+1} \big(q(\bar{x}) \rightarrow p(\bar{x}, x_{k+1})\big)$
so there is $b_{k+1} \in M_\mbC$ such that $\mcM_\mbC \models p(\bar{b}, b_{k+1})$.
Hence $(\mcM, \bar{a}a_{k+1}) \equiv_{\cl} (\mcM_\mbC, \bar{b}b_{k+1})$
so the claim is proved.

\medskip

\noindent
By the claim we conclude that $\mcM \cong \mcM_\mbC$ and hence $T = T_\mbC$.
We now recall that, by the choice of $\mbC$, we have 
\begin{enumerate}
\item $\mbC$ satisfies 
Assumptions~\ref{assumptions on the closure},~\ref{assumption about a unique 0-closure},
and~\ref{stronger assumption that nonclosed types have realizations}
\end{enumerate}
and, by the choice of $T_\mbC$,
\begin{enumerate}
\item[(2)] $T_\mbC$ is the set of all sentences $\varphi$ such that $\varphi$ is true in every 
sufficiently large member of $\mbC$.
\end{enumerate}
Since we have proved that $T = T_\mbC$ it follows that~(2) holds if $T_\mbC$ is replaced by $T$.
This completes the proof of part~(i) of Theorem~\ref{characterization of fmp}.

Above we concluded that for every formula $\varphi(\bar{x})$ that is consistent with $T$ there is a
disjunction of closure types to which $\varphi(\bar{x})$ is equivalent in all sufficiently large members of $\mbC$.
It follows that $\varphi(\bar{x})$ is equivalent to this disjunction of closure types in all models of $T_\mbC$.
As $T = T_\mbC$ we have also proved part~(ii) of Theorem~\ref{characterization of fmp}.
\hfill $\square$

\begin{rem}\label{remark on FSMP and C}{\rm
Given Theorem~\ref{HPC and fsmp} one can ask the following variation of
Theorem~\ref{characterization of fmp}
is possible? Is FSMP of a complete $\omega$-categorical theory $T$ equivalent to the existence of 
a class $\mbC$ of finite structures such that, 
assumptions~\ref{assumptions on the closure},~\ref{assumption about a unique 0-closure},
and~\ref{stronger assumption that nonclosed types have realizations} hold, $\mbC$ has HPC,
and $T$ is the set of all sentences that are true in all sufficiently large membersof $\mbC$?
Theorem~\ref{HPC and fsmp} gives that the existence of such $\mbC$ implies that $T$ has FSMP,
but I have not been able to prove the other direction although it first seemed like the proof of
Theorem~\ref{HPC and fsmp} could easily be modified for this purpose.
}\end{rem}

\section{Independence in a class of finite structures}\label{independence in C}

\noindent
In this section we assume that $\mbC$ is a class of finite structures such that
assumptions~\ref{assumptions on the closure},
\ref{assumption about a unique 0-closure}
and~\ref{stronger assumption that nonclosed types have realizations}
are satisfied.
Hence also (the weaker) Assumption~\ref{assumption that nonclosed types have realizations} holds.
Then Corollary~\ref{0-1 law} implies that for every first-order $\tau$-sentence $\varphi$ either $\varphi$ is true in all 
sufficiently large $\mcA \in \mbC$ or $\varphi$ is false in all sufficiently large $\mcA \in \mbC$.
{\em Let $T$ be the set of all sentences $\varphi$ such that $\varphi$ is true in every sufficiently large 
$\mcA \in \mbC$},
so $T$ is a consistent and complete theory.
By Theorem~\ref{T is omega-categorical},
$T$ is $\omega$-categorical.
By Proposition~\ref{cl and acl coincide in limit structures}, for all $\mcM \models T$, $k \in \mbbN$,  and
$a_1, \ldots, a_{k+1} \in M$ we have $a_{k+1} \in \acl_\mcM(a_1, \ldots, a_k)$
if and only if $\mcM \models \xi_k(a_1, \ldots, a_{k+1})$.

As shown in \cite{KopInd}, since  $T$ is $\omega$-categorical there is, for every $n \in \mbbN$, 
a notion of $n$-rank which
gives rise to a notion of $n$-independence in models of $T$ which we denote by $\indn$.
Technically speaking, for every $\mcM \models T$,
$\indn$ is a ternary relation on the set of subsets of $M$ that have ``small'' cardinality compared to $|M|$.
Finite subsets of $M$ are always considered small compared to $|M|$ since $M$ is infinite.
Thus $\indn$ is, in particular, a ternary relation on the set of finite subsets of $M$.
The notation $A \underset{C}{\indn} B$ means that the relation $\indn$ is satisfied by the triple $(A, B, C)$.
For sequences $\bar{a}$, $\bar{b}$, and $\bar{c}$, $\bar{a} \underset{\bar{c}}{\indn} \bar{b}$
means that $\rng(\bar{a}) \underset{\rng(\bar{c})}{\indn} \rng(\bar{b})$.

\begin{fact}\label{properties of n-independence} \cite{KopInd}
For every $n \in \mbbN$ and $\mcM \models T$ the following hold for all ``small'' $A, B, C, D \subseteq M$:
\begin{enumerate}

\item (invariance) If $\sigma$ is an elementary function (i.e. one that preserves the truth value of
first-order formulas) from some subset of $M$
that includes $ABC$ to some subset of $M$,
then $A \underset{C}{\indn} B$ if and only if $\sigma(A) \underset{\sigma(C)}{\indn} \sigma(B)$.

\item (local character) If $A$ is finite then there is a countable $C \subseteq B$ such that
$A \underset{C}{\indn} B$.

\item (finite character) $A \underset{C}{\indn} B$ if and only if for all finite $B' \subseteq B$,
$A \underset{C}{\indn} B'$.

\item (extension) If $k < \omega$, 
$A = \{a_1, \ldots, a_k\}$, and $C \subseteq B$, then there is $A' = \{a'_1, \ldots, a'_k\} \subseteq M$ such that
$\tp(a'_1, \ldots, a'_k / C) = \tp(a_1, \ldots, a_k / C)$ and $A' \underset{C}{\indn} B$.

\item (monotonicity)
 If $B \subseteq C \subseteq D$ and
$A \underset{B}{\indn} D$ then $A \underset{B}{\indn} C$ and $A \underset{C}{\indn} D$.

\item (transitivity) If $B \subseteq C \subseteq D$,
 $A \underset{B}{\indn} C$ and $A \underset{C}{\indn} D$, then $A \underset{B}{\indn} D$.

\end{enumerate}
\end{fact}

\noindent
Since $T$ is $\omega$-categorical and $\indn$ is invariant under elementary maps it follows that,
for all $i, j, k \in \mbbN$ where $i, j > 0$, there is a formula $\theta_{i, j, k}^n(\bar{x}, \bar{y}, \bar{z})$,
where $|\bar{x}| = i$, $|\bar{y}| = j$, and $|\bar{z}| = k$ such that for all $\mcM \models T$ and
$\bar{a} \in M^i$, $\bar{b} \in M^j$, and $\bar{c} \in M^k$,
\[
\bar{a} \underset{\bar{c}}{\indn} \bar{b} \quad \text{ if and only if } \quad 
\mcM \models \theta_{i, j, k}^n(\bar{a}, \bar{b}, \bar{c}).
\]

\subsection{Transfering $\indn$ from models of $T$ to members of $\mbC$}

Since $T$ is $\omega$-categorical it follows that for every $m \in \mbbN$ there is a sentence $\chi_m$
which (in all models of $T$) expresses that properties (2)--(6) of  Fact~\ref{properties of n-independence} hold
for all $A, B, C, D$ of cardinality at most $m$.
By the choice of $T$, for all $m \in \mbbN$,  $\chi_m$ is true in all sufficiently large members of $\mbC$.
For every $\mcS\in \mbC$ and sequences $\bar{a}, \bar{b}, \bar{c}$ of length at most $m$ of elements from $S$, 
define
\[
\bar{a} \underset{\bar{c}}{\indn} \bar{b} \quad \text{ if and only if } \quad
\mcS \models \theta_{|\bar{a}|, |\bar{b}|, |\bar{c}|}^n(\bar{a}, \bar{b}, \bar{c}).
\]
It follows that, for every $m \in \mbbN$, $\indn$ has all properties (2)--(6) of Fact~\ref{properties of n-independence}
in all sufficiently large $\mcS \in \mbC$ for all $A, B, C, D \subseteq S$ of cardinality at most $m$.

Ideally, we would like that $\indn$ also has the following property:
\begin{enumerate}
\item[(7)] (symmetry) If $A \underset{C}{\indn} B$ then $B \underset{C}{\indn} A$.
\end{enumerate}
The results in \cite{KopInd} do not guarantee that $\indn$ is symmetric in $M$, but they give a characterization of when it is. 
An adaptation of those results 
(derived from the proof of \cite[Proposition~5.13]{KopInd})
to the case when we consider sets with a common finite bound on the cardinality gives the 
next fact which uses the notion of $n$-rank from \cite{KopInd}.
For all $n \in \mbbN$ and finite sequences $\bar{a}$ and $\bar{b}$ of elements from $M$, 
let $\rk^n(\bar{a} / \bar{b})$ denote the {\em $n$-rank of $\bar{a}$ over $\bar{b}$}
(which will be explained below),
which is always a natural number for a finite sequence $\bar{a}$ as shown in \cite{KopInd}.

\begin{fact}\label{characterization of symmetry} \cite{KopInd}
Let $n \in \mbbN$.
For every $m \in \mbbN$ there is $\kappa(m) \in \mbbN$ such that 
condition~($e_{\kappa(m)}$) below implies condition~($s_m$):
\begin{enumerate}
\item[($e_{\kappa(m)}$)] (exchange property up to $\kappa(m)$) For every $\mcM \models T$, all
$C \in \mcP_{\kappa(m)}(M)$, $2 \leq k \leq \kappa(m)$, and $a_1, \ldots, a_k \in M$,
if $\rk^n(a_i / C) = 1$ for all $i = 1, \ldots, k$, and
$a_k \in \acl(\{a_1, \ldots, a_{k-1}\} \cup C) \setminus \acl(\{a_2, \ldots, a_{k-1}\} \cup C)$,
then $a_1 \in \acl(\{a_2, \ldots, a_k\} \cup C)$.

\item[($s_m$)] (symmetry up to $m$) For every $\mcM \models T$ and all $A, B, C \in \mcP_m(M)$, 
if $A \underset{C}{\indn} B$ then $B \underset{C}{\indn} A$.
\end{enumerate}
\end{fact}

\noindent
By the $\omega$-categoricity of $T$ there is, for all $i, j, n, \alpha \in \mbbN$, a formula 
$\rho_{i, j}^{n, \alpha}(\bar{x}, \bar{y})$
with $|\bar{x}| = i$ and $|\bar{y}| = j$ such that for all $\mcM \models T$, $\bar{a} \in M^i$ and $\bar{b} \in M^j$,
\[
\rk^n(\bar{a} / \bar{b}) = \alpha \quad \text{ if and only if } \quad 
\mcM \models \rho_{i, j}^{n, \alpha}(\bar{a}, \bar{b}).
\]
Now property~($e_{\kappa(m)}$) of Fact~\ref{characterization of symmetry}
is expressed by the conjunction of all sentences of this form where $k, l \leq \kappa(m)$:
\begin{equation*}
\begin{aligned}
&\forall x_1, \ldots, x_k, y_1, \ldots, y_l \bigg( \bigg[
\bigwedge_{i=1}^k \rho_{1, l}^{n, 1}(x_i, y_1, \ldots, y_l)  \\ 
&\wedge \ 
\xi_{k-1 + l}(x_1, \ldots, x_{k-1}, y_1, \ldots, y_l, x_k) \\ 
&\wedge \
\neg\xi_{k-2 + l}(x_2, \ldots, x_{k-1}, y_1, \ldots, y_l, x_k) \bigg] \\
&\longrightarrow
\xi_{k-1 + l}(x_2, \ldots, x_{k}, y_1, \ldots, y_l, x_1) \bigg).
\end{aligned}
\end{equation*}
By using the formulas $\theta_{i, j, k}^n$ with $i, j, k \leq m$ it is easy to see that condition~($s_m$)
of Fact~\ref{characterization of symmetry}
can be expressed by a first-order sentence.
Hence there is a first-order sentence which expresses that condition~($e_{\kappa(m)}$) implies condition~($s_m$).
From the definition of $T$ it follows that this sentence is true in all sufficiently large members of $\mbC$
and this implies the following:

\begin{lem}\label{exchange property and symmetry in C}
Condition~($e^\mbC_{\kappa(m)}$) below implies~($s_m^\mbC$):
\begin{enumerate}
\item[($e^\mbC_{\kappa(m)}$)] 
{\rm (exchange property up to $\kappa(m)$ in $\mbC$)} For all sufficently large $\mcS \in \mbC$,
$C = \{c_1, \ldots, c_l\} \subseteq S$ with $l \leq \kappa(m)$, $2 \leq k \leq \kappa(m)$, and $a_1, \ldots, a_k \in M$,
if $\mcS \models \rho_{1, l}^{n, 1}(a_i, c_1, \ldots, c_l)$ for all $i = 1, \ldots, k$, and
$a_k \in \cl_\mcS(\{a_1, \ldots, a_{k-1}\} \cup C) \setminus \cl_\mcS(\{a_2, \ldots, a_{k-1}\} \cup C)$,
then $a_1 \in \cl_\mcS(\{a_2, \ldots, a_k\} \cup C)$.

\item[($s_m^\mbC$)] {\rm (symmetry up to $m$ in $\mbC$)} For all sufficiently large $\mcS \in \mbC$ 
and all $A, B, C \in \mcP_m(S)$, 
if $A \underset{C}{\indn} B$ then $B \underset{C}{\indn} A$.
\end{enumerate}
\end{lem}

So far we have not explained what $A \underset{C}{\indn} B$ means for subsets of a model $\mcM \models T$
(or what it means in structures in $\mbC$).
By \cite[Definition~4.1]{KopInd}, for finite $A, B, C \subset M$ we have
\[
A \underset{C}{\indn} B \quad \text{ if and only if } \quad \rk^n(A / BC) = \rk^n(A / C).
\]
So to explain what $\indn$ means it now suffices to explain what $\rk^n$ means.
For any positive $n \in \mbbN$, let $M_n$ denote the subset of $M\meq$ that corresponds every
imaginary element that corresponds to an equivalence class of a $\es$-definable equivalence relation on $M^k$
for some $k \leq n$.
Let $\acl\meq$ denote the algebraic closure operator in $\mcM\meq$.
By \cite[Lemma~3.5]{KopInd}, for $\alpha \in \mbbN$ and finite sequences $\bar{a}$ and $\bar{b}$
of elements from $M$ we have
\begin{enumerate}
\item[] $\rk^n(\bar{a} / \bar{b}) \geq \alpha$ if and only if there are 
$a'_1, \ldots, a'_\alpha \in M_n$ such that for all $i = 1, \ldots, \alpha$,
$a'_i \in \acl\meq(\bar{a})$ and 
$a'_i \notin \acl\meq(a'_1, \ldots, a'_{i-1}, \bar{b})$.
\end{enumerate}
By \cite[Definition~3.2]{KopInd}, 
\begin{enumerate}
\item[] $\rk^n(\bar{a} / \bar{b}) = \alpha$ if and only if 
$\rk^n(\bar{a} / \bar{b}) \geq \alpha$ and $\rk^n(\bar{a} / \bar{b}) \not\geq \alpha+1$.
\end{enumerate}
As $T$ is $\omega$-categorical and $\bar{a}$ is a finite sequence, 
$\acl\meq(\bar{a}) \cap M_n$ is finite and its cardinality is bounded by some $\beta \in \mbbN$ 
that depends only on $l = |\bar{a}|$
and $T$. It follows that for this $\beta$, $\rk^n(\bar{a} / \bar{b}) \leq \beta$ for all $\bar{a}$ of length at most $l$
and for all $\bar{b}$ (of any length).
By \cite[Lemma~3.9]{KopInd} we have $\rk^n(\bar{a} / \bar{b}\bar{c}) \leq \rk^n(\bar{a} / \bar{b})$
for all finite sequences $\bar{a}$, $\bar{b}$, and $\bar{c}$.

\subsection{The case $n = 0$}

It follows from the characterization of $\rk^n$ above that if we let $\rho_{i, j}^{0, \geq \alpha}(\bar{x}, \bar{y})$, 
where $i = |\bar{x}|$ and $j = |\bar{y}|$, be the formula
\[
\exists z_1, \ldots, z_\alpha \bigg(
\bigwedge_{l=1}^\alpha \bigg[
\xi_i(\bar{x}, z_l) \ \wedge \ 
\neg\xi_{j+l-1}(\bar{y}, z_1, \ldots, z_l) \bigg] \bigg)
\]
then, for all $\bar{a} \in M^i$,  $\bar{b} \in M^j$ and $\alpha \in \mbbN$,
\[
\rk^0(\bar{a} / \bar{b}) \geq \alpha \quad \Longleftrightarrow \quad
\mcM \models \rho_{i, j}^{0, \geq \alpha}(\bar{a}, \bar{b})
\]
and 
\[
\rk^0(\bar{a} / \bar{b}) = \alpha \quad \Longleftrightarrow \quad
\mcM \models \rho_{i, j}^{0, \geq \alpha}(\bar{a}, \bar{b})
\wedge \neg\rho_{i, j}^{0, \geq \alpha+1}(\bar{a}, \bar{b}).
\]
Let $i, j, k \in \mbbN$ and let $\beta_i \in \mbbN$ be such that $\rk^n(\bar{a}) \leq \beta_i$ whenever $\bar{a} \in M^i$.
Now it follows that if $\bar{a} \in M^i$, $\bar{b} \in M^j$, and $\bar{c} \in M^k$, then
\begin{align*}
&\mcM \models \theta_{i, j, k}^0(\bar{a}, \bar{b}, \bar{c})  \ \Longleftrightarrow \\
&\bar{a} \underset{\bar{c}}{\indo} \bar{b} \ \Longleftrightarrow \\ 
&\mcM \models \bigwedge_{\gamma=1}^{\beta_i} \Big(\rho_{i, k}^{0, \geq \gamma}(\bar{a}, \bar{c})
\rightarrow \rho_{i, j+k}^{0, \geq \gamma}(\bar{a}, \bar{b}, \bar{c}) \Big).
\end{align*}
Hence
\[
\mcM \models \lambda_{i, j, k}^0 := 
\forall \bar{x}, \bar{y}, \bar{z} \bigg(
\theta_{i, j, k}^0(\bar{x}, \bar{y}, \bar{z}) \ \longleftrightarrow \ 
\bigwedge_{\gamma=1}^\beta \Big(\rho_{i, k}^{0, \geq \gamma}(\bar{x}, \bar{z})
\rightarrow \rho_{i, j+k}^{0, \geq \gamma}(\bar{x}, \bar{y}, \bar{z}) \Big) \bigg).
\]
Thus $\lambda_{i, j, k}^0$ is true in all sufficiently large members of $\mbC$ which 
gives the following with $\beta_i$ as above:

\begin{lem}\label{characterization of independence in C}
If $\mcA \in \mbC$
is sufficiently large compared to $i, j$ and $k$, $\bar{a} \in A^i$, $\bar{b} \in A^j$, and $\bar{c} \in A^k$, then
\begin{align*}
&\bar{a} \underset{\bar{c}}{\nindo} \bar{b} \ \Longleftrightarrow \ 
\neg \theta_{i, j, k}^0(\bar{a}, \bar{b}, \bar{c}) \ \Longleftrightarrow \\
&\text{ for some positive } \ \alpha \leq \beta_i, \ \exists a'_1, \ldots, a'_\alpha \in A, \\
&\bigwedge_{l=1}^\alpha 
\Big(a'_l \in \cl_\mcA(\bar{a}) \ \wedge \ 
a'_l \notin \cl_\mcA(\bar{b}, a'_1, \ldots, a'_{l-1})\Big) \ \text{ and } \\
& \neg \exists a''_1, \ldots, a''_\alpha \in A, \\
&\bigwedge_{l=1}^\alpha 
\Big(a''_l \in \cl_\mcA(\bar{a}) \ \wedge \ 
a''_l \notin \cl_\mcA(\bar{b}, \bar{c}, a''_1, \ldots, a''_{l-1})\Big).
\end{align*}
\end{lem}

\noindent
Hence, for any fixed $i \in \mbbN$ the notion $\indo$ restricted to sequences
of length at most $i$ is defined in terms of $\cl_\mcA$ (as shown above) in all sufficiently large $\mcA \in \mbC$.

\begin{exam} {\rm
Let $\mbC$ be as in Example~\ref{continuation of example 1}
with $s = 1$, so all members of $\mbC$ are trees in which all leaves belong to level $t$.
For any fixed $i \in \mbbN$, if $\mcA \in \mbC$ is sufficiently large and $\bar{a}$, $\bar{b}$, and $\bar{c}$
are sequences of elements from $A$ of length at most $i$, then
$\bar{a} \underset{\bar{c}}{\indo} \bar{b}$ if and only if 
$\cl_\mcA(\bar{a}) \setminus \cl_\mcA(\bar{c}) = \cl_\mcA(\bar{a}) \setminus \cl_\mcA(\bar{b}C)$.
}\end{exam}

\subsection{The case $n > 0$}

Now suppose that $n > 0$.

If, for some $0 < m \in \mbbN$, $E$ is a $\es$-definable equivalence relation on $M^m$
and $\bar{a} \in M^m$, then $[\bar{a}]_E$ denotes the $E$-equivalence class of $\bar{a}$ and we also
view $[\bar{a}]_E$ as an imaginary element.

\begin{lem}\label{characterization of acl for imaginaries}
For $i = 1, \ldots, k$, let $E_i$ be a $\es$-definable equivalence relation on $M^{l_i}$.
For $\bar{a}_i \in M^{l_i}$, $i = 1, \ldots, k$, we have 
\[
[\bar{a}_{k+1}]_{E_{k+1}} \in \acl\meq\big([\bar{a}_1]_{E_1}, \ldots, [\bar{a}_k]_{E_k}\big)
\ \ \text{ if and only if}
\]
there is $m \in \mbbN$ such that there do {\rm not} exist 
$\bar{b}_1, \ldots, \bar{b}_m$ and, for $i = 1, \ldots, m$, $\bar{c}_{i, 1}, \ldots, \bar{c}_{i, k}$
such that 
$\tp(\bar{c}_{i, 1}, \ldots, \bar{c}_{i, k}, \bar{b}_i) = \tp(\bar{a}_1, \ldots, \bar{a}_k, \bar{a}_{k+1})$,
$\mcM \models \neg E(\bar{b}_i, \bar{b}_j)$ for $i \neq j$, and
$\mcM \models E(\bar{c}_{i, j}, \bar{a}_j)$ for all $i$ and $j$.
Moreover, the number $m$ only depends on $T$ and $\tp(\bar{a}_1, \ldots, \bar{a}_k, \bar{a}_{k+1})$.
\end{lem}

\noindent
{\bf Proof.}
It suffices to prove that
\[
[\bar{a}_{k+1}]_{E_{k+1}} \notin \acl\meq\big([\bar{a}_1]_{E_1}, \ldots, [\bar{a}_k]_{E_k}\big)
\]
if and only if
for all $m \in \mbbN$ there are
$\bar{b}_1, \ldots, \bar{b}_m$ and, for $i = 1, \ldots, m$, $\bar{c}_{i, 1}, \ldots, \bar{c}_{i, k}$
such that 
$\tp(\bar{c}_{i, 1}, \ldots, \bar{c}_{i, k}, \bar{b}_i) = \tp(\bar{a}_1, \ldots, \bar{a}_k, \bar{a}_{k+1})$,
$\mcM \models \neg E(\bar{b}_i, \bar{b}_j)$ for $i \neq j$, and
$\mcM \models E(\bar{c}_{i, j}, \bar{a}_j)$ for all $i$ and $j$.
If 
\(
[\bar{a}_{k+1}]_{E_{k+1}} \notin \acl\meq\big([\bar{a}_1]_{E_1}, \ldots, [\bar{a}_k]_{E_k}\big)
\)
then, for any $m \in \mbbN$, there are distinct imaginaries $a'_1, \ldots, a'_m \in M\meq$ such that (in $\mcM\meq$)
\[
\tp([\bar{a}_1]_{E_1}, \ldots, [\bar{a}_k]_{E_{k+1}}) = 
\tp([\bar{a}_1]_{E_1}, \ldots, [\bar{a}_k]_{E_k}, a'_i)
\] 
for all $i = 1, \ldots, m$.
By properties of $\mcM\meq$ the existence of 
$\bar{b}_1, \ldots, \bar{b}_m$ and, for $i = 1, \ldots, m$, $\bar{c}_{i, 1}, \ldots, \bar{c}_{i, k}$
as above follows.

Conversely, if for all $m \in \mbbN$ there are 
$\bar{b}_1, \ldots, \bar{b}_m$ and, for $i = 1, \ldots, m$, $\bar{c}_{i, 1}, \ldots, \bar{c}_{i, k}$
such that $\tp(\bar{c}_{i, 1}, \ldots, \bar{c}_{i, k}, \bar{b}_i) = \tp(\bar{a}_1, \ldots, \bar{a}_k, \bar{a}_{k+1})$,
$\mcM \models \neg E(\bar{b}_i, \bar{b}_j)$ for $i \neq j$, and
$\mcM \models E(\bar{c}_{i, j}, \bar{a}_j)$ for all $i$ and $j$, then clearly
$[\bar{a}_{k+1}]_{E_{k+1}} \notin  \acl\meq\big([\bar{a}_1]_{E_1}, \ldots, [\bar{a}_k]_{E_k}\big)$.
\hfill $\square$

\medskip

\noindent
For all $k, l_1, \ldots, l_{k+1} \in \mbbN$, let 
$p_1(\bar{x}_1, \ldots, \bar{x}_{k+1}), \ldots, p_s(\bar{x}_1, \ldots, \bar{x}_{k+1})$,
where $|\bar{x}_i| = l_i$, be an enumeration of all (up to equivalence) formulas that isolate
complete types over $\es$ in the variables $\bar{x}_1 \ldots \bar{x}_{k+1}$.
If, for $i = 1, \ldots, k+1$, $E_i$ is a $\es$-definable equivalence relation on $M^{l_i}$
and $m \in \mbbN$, we let 
$\xi_{\bar{E}}^m(\bar{x}_1, \ldots, \bar{x}_{k+1})$, where $\bar{E} = (E_1, \ldots, E_{k+1})$, denote the formula
\begin{align*}
&\bigwedge_{t=1}^s \bigg(
p_t(\bar{x}_1, \ldots, \bar{x}_{k+1}) \rightarrow 
\neg \exists \bar{u}_1, \ldots, \bar{u}_{m}, \bar{v}_{1, 1}, \ldots, \bar{v}_{1, k}, \ldots,
\bar{v}_{m, 1}, \ldots, \bar{v}_{m, k} 
\bigg[ \\
&\bigwedge_{i=1}^m 
p_t(\bar{v}_{i, 1}, \ldots, \bar{v}_{i, k}, \bar{u}_i) \ \wedge \ 
\bigwedge_{\substack{1 \leq i, j \leq m \\ i \neq j}} \neg E_{k+1}(\bar{u}_i, \bar{u}_j) \ \wedge \ 
\bigwedge_{i=1}^m \bigwedge_{j=1}^k E_j(\bar{v}_{i, j}, \bar{x}_j) \bigg] \bigg).
\end{align*}

\noindent
From Lemma~\ref{characterization of acl for imaginaries}
it follows that we have
\begin{enumerate}
\item[] $[\bar{a}_{k+1}]_{E_{k+1}} \in \acl\meq\big([\bar{a}_1]_{E_1}, \ldots, [\bar{a}_k]_{E_1}\big)$
if and only if there is $m \in \mbbN$ 
such that
$\mcM \models \xi_{\bar{E}}^m(\bar{a}_1, \ldots, \bar{a}_{k+1})$
(where $m$ depends only on $T$ and $\tp(\bar{a}_1, \ldots, \bar{a}_k, \bar{a}_{k+1})$).
\end{enumerate}

\noindent
Observe that if $E_i$ is the equivalence relation `$=$' on $M$  for all $i = 1, \ldots, k+1$
(so that we can let $\bar{x}_i$ be $x_i$ for all $i$),
then $\xi_{\bar{E}}(x_1, \ldots, x_{k+1})$ is equivalent (in models of $T$) to $\xi_k(x_1, \ldots, x_{k+1})$.

From the characterization of `$\rk^n(\bar{a} / \bar{b}) \geq \alpha$' above, where $\bar{a}$ and $\bar{b}$
are finite sequences of elements from $M$, it now follows that
\begin{enumerate}
\item[]
$\rk^n(\bar{a} / \bar{b}) \geq \alpha$ if and only if 
there are $\es$-definable equivalence relations $E_1, \ldots, E_\alpha$ (on $M^{l_i}$ for $i = 1, \ldots, \alpha$)
and $\bar{c}_1, \ldots, \bar{c}_\alpha$ such that 
\[
\mcM \models \bigwedge_{l=1}^\alpha \bigg(
\xi_{(=, \ldots, =, E_l)}(\bar{a}, \bar{c}_l) \ \wedge \ 
\neg \xi_{(=, \ldots, =, E_1, \ldots, E_l)}(\bar{b}, \bar{c}_1, \ldots, \bar{c}_l) \bigg)
\]
\end{enumerate}
Let $\rho_{i, j}^{n, \geq \alpha}(\bar{x}, \bar{y})$, where $|\bar{x}| = i$ and $|\bar{y}| = j$,
be the disjunction of all formulas as above where $E_1, \ldots, E_\alpha$ range
over all $\es$-definable equivalence relations on $M^s$ for $1 \leq s \leq n$.
Then, for $\bar{a} \in M^i$ and $\bar{b} \in M^j$,
\[
\rk^n(\bar{a} / \bar{b}) \geq \alpha \ \Longleftrightarrow \ \mcM \models \rho_{i, j}^{n, \geq \alpha}(\bar{a}, \bar{b})
\]
and 
\[
\rk^n(\bar{a} / \bar{b}) = \alpha \ \Longleftrightarrow \ 
\mcM \models \rho_{i, j}^{n, \geq \alpha}(\bar{a}, \bar{b}) \wedge
\neg\rho_{i, j}^{n, \geq \alpha+1}(\bar{a}, \bar{b}).
\]
Recall that $\bar{a} \underset{\bar{c}}{\nindn} \bar{b}$ if and only if
for some $\alpha \in \mbbN$, $\rk^n(\bar{a} / \bar{b}) \geq \alpha$ and 
$\rk^n(\bar{a} / \bar{b}\bar{c}) \not\geq \alpha$.
Moreover, this $\alpha$ is smaller than some $\beta \in \mbbN$ that is determined only by $T$
and $|\bar{a}|$.
Since every first-order sentence which is true in $\mcM$ is also true in all sufficiently large members of $\mbC$ we
get the following:

\begin{lem}\label{characterization of independence in C}
For all $i, j, k \in \mbbN$ with $i, j > 0$, sufficiently large $\mcA \in \mbC$ and finite sequences
 $\bar{a} \in A^i$, $\bar{b} \in A^j$, and $\bar{c} \in A^k$,
$\bar{a} \underset{\bar{c}}{\nindn} \bar{b}$ if and only if there is $\alpha \in \mbbN$ such that
$\mcA \models \rho_{i, j}^{n, \geq \alpha}(\bar{a}, \bar{b}) \wedge 
\neg\rho_{i, j+k}^{n, \geq \alpha}(\bar{a}, \bar{b}, \bar{c})$
(where $\alpha$ is smaller than some number that is determined only by $T$ and $i$).
\end{lem}

\noindent
Since the formulas of the form $\rho_{i, j}^{n, \geq \alpha}(\bar{x}, \bar{y})$ have been explicitly described
the above lemma explains the meaning of $\indn$ in structures in $\mbC$.

\end{document}